\documentclass[11pt]{amsart}
\usepackage[english]{babel}
\usepackage{graphics}
\usepackage{amsfonts}
\usepackage{amsmath}
\usepackage{amssymb}
\usepackage{mathrsfs}
\usepackage{enumerate}
\usepackage{relsize,exscale}
\usepackage{cite}
\usepackage[colorlinks=true,urlcolor=blue,
citecolor=red,linkcolor=blue,linktocpage,pdfpagelabels,
bookmarksnumbered,bookmarksopen]{hyperref}
\usepackage[hyperpageref]{backref}
\usepackage{esint}
\usepackage{epsfig}
\usepackage{amscd}
\usepackage{amsmath}
\usepackage{graphicx}
\usepackage{newlfont}
\usepackage{latexsym}
\usepackage{amsthm}

\usepackage{color}

\newtheorem{theorem}{Theorem}[section]
\newtheorem{lemma}[theorem]{Lemma}

\newtheorem{proposition}[theorem]{Proposition}

\theoremstyle{remark}
\newtheorem{remark}[theorem]{Remark}

\numberwithin{equation}{section}

\theoremstyle{definition}
\newtheorem{definition}[theorem]{Definition}

\newcommand{\ep}{\epsilon}

\newcommand{\R}{\mathbb{R}}

\newcommand{\N}{\mathbb{N}}

\newcommand{\scal}[2]{\langle {#1} , {#2}\rangle}

\newcommand{\Sl}{\Sigma_{\lambda}}
\newcommand{\Spl}{\Sigma'_{\lambda}}

\newcommand{\Slo}{\Sigma_{\lambda_1}}

\makeindex
\begin{document}

\title{Symmetry and rigidity for the hinged composite plate problem}

\author[F.\ Colasuonno]{Francesca Colasuonno}
\author[E.\ Vecchi]{Eugenio Vecchi}
\address[F.\ Colasuonno]{Dipartimento di Matematica ``Giuseppe Peano''\newline\indent
	Universit\`a degli Studi di Torino \newline\indent
	Via Carlo Alberto 10, 10123, Torino, Italy}
\email{francesca.colasuonno@unito.it}

\address[E.\ Vecchi]{Dipartimento di Matematica ``Guido Castelnuovo''\newline\indent
	Sapienza Universit\`a di Roma \newline\indent
	P.le Aldo Moro 5, 00185, Roma, Italy}
\email{vecchi@mat.uniroma1.it}
\thanks{{\bf Acknowledgements.} The authors warmly thank the anonymous referee for correcting a
mistake in the original version of the paper and for the suggestions that improved the paper.\\
The authors are grateful to Prof. Sagun Chanillo for having suggested the problem.
The authors thank also Prof. Bruno Franchi for the interesting discussions on the subject.\\
F.C. and E.V. are supported by {\em Gruppo Nazionale per l'Analisi Ma\-te\-ma\-ti\-ca, la Probabilit\`a e le loro Applicazioni} (GNAMPA) of the {\em Istituto Nazionale di Alta Matematica} (INdAM). E.V. is partially supported by the INdAM-GNAMPA Project 2018 ``Problemi di curvatura relativi ad operatori ellittico-degeneri''.}



\date{\today}

\begin{abstract}
The composite plate problem is an eigenvalue optimization
problem related to the fourth order operator $(-\Delta)^2$.
In this paper we continue the study started in \cite{CoVe}, 
focusing on symmetry and rigidity
issues in the case of the {\it hinged composite plate problem},
a specific situation that allows us to exploit classical techniques
like the moving plane method.
\end{abstract}

\maketitle


\section{Introduction}\label{Intro}
The composite plate problem is an eigenvalue optimization problem
that extends to the fourth order case the composite membrane problem
extensively studied e.g., in  \cite{CGIKO00,CGK,CK08,CKT,Sha,Chanillo13,CupVe}, see also \cite{HZ,TZ} and references therein for related problems. 
One of the interesting situations is provided by
the so called hinged composite plate, which can be described as follows: build a hinged plate, of prescribed mass and shape, out of different materials sharing the same elastic properties but having different densities, in such a way that its principal frequency is as low as possible.  

We introduce below the above problem from a mathematical point of view. Let $\Omega \subset \R^n$ be a bounded domain,  
$0<h < H$ be two positive constants, and $M \in [h |\Omega|, H |\Omega|]$.
We define the set of admissible densities as 
\begin{equation*}
\mathrm{P} := \left\{ \rho:\Omega \to \R: \, \int_{\Omega}\rho(x)\,dx=M, \,\, h \leq \rho \leq H \, \textrm{in } \Omega \right\}.
\end{equation*}
By {\it hinged composite plate problem} we mean the following minimization problem
\begin{equation}\label{Comp}
\Theta(h,H,M) := \inf_{\rho \in \mathrm{P}}\;\; \inf_{u\in H^2(\Omega)\cap H^{1}_{0}(\Omega)\setminus\{0\}} \frac{\mathop{\mathlarger{\int_{\Omega}}}(\Delta u)^2}{\mathop{\mathlarger{\int_{\Omega}}}\rho \, u^2}.
\end{equation}
A couple $(u,\rho)$ which realizes the double infimum in \eqref{Comp}
is called {\it optimal pair}. \\
To the best of our knowledge, the first results concerning problem \eqref{Comp} have been obtained
in \cite{CoVe}. 
The Euler-Lagrange equation associated to \eqref{Comp} gives rise to the following fourth order Navier problem 
\begin{equation}\label{P}
\left\{ \begin{array}{rl}
          \Delta^2 u = \Theta \rho  u & \quad \textrm{in } \Omega,\\
					u= \Delta u =0 & \quad \textrm{on } \partial \Omega.\\
					\end{array}\right.
\end{equation}
A first crucial aspect when dealing with fourth order PDE's like \eqref{P} is the
positivity of the solutions. This issue is closely related to the validity, or
non validity, of a maximum principle, and it is one of the most studied topics
in the literature of higher order PDE's. We refer to the monograph \cite{GGS} for
a wide and comprehensive treatment of the subject, cf. also \cite{PucRad,PS2} and references therein for more results on polyharmonic problems, and \cite{Birindelli,Sirakov} and references therein for maximum principles for cooperative elliptic systems.\\
In our case the maximum principle holds essentially thanks to the fact that we work with Navier boundary conditions, and the proof of the positivity is slightly simplified because  the solutions of \eqref{P} we are considering are {\it minimizers} of the variational problem \eqref{Comp}, cf. also \cite[Proposition 5.1]{CoVe}.
This enables us to assume --without loss of generality-- that if $(u,\rho)$ is an optimal pair, $u$ is of fixed sign, say positive, under suitable assumptions on the domain $\Omega$, cf. Proposition~\ref{pos}.
Furthermore, in \cite[Theorems 1.3 and 1.4]{CoVe}, we give an explicit representation of the optimal configuration $\rho$, namely if $(u,\rho)$ is an optimal pair, 
\begin{equation}\label{rho}
\rho (x) = \rho_u (x) := h \, \chi_{\{u \leq t\}}(x) + H \, \chi_{\{u >t\}}(x)\quad	\mbox{for all }x\in\Omega
\end{equation}
for a suitable $t=t(h,H,\Omega,u)>0$, and we prove
that, if $\partial \Omega$ is smooth enough, every optimal pair $(u,\rho)$ satisfies 
\begin{equation}\label{reg}
u \in C^{3,\alpha}(\overline{\Omega}) \cap W^{4,q}(\Omega) \quad \textrm{for all } \alpha \in (0,1) \textrm{ and } q \geq 1.
\end{equation}
In particular, the combination of the regularity of $u$ up to the boundary
and the boundary condition $u=0$ on $\partial \Omega$, yields that the set $\{u \leq t\}$ contains a tubular
neighborhood of $\partial \Omega$.
We point out that the existence of optimal pairs as proved in \cite{CoVe} 
does not rely on the regularity of the boundary $\partial \Omega$. The only
part where the regularity of the boundary is truly exploited is to get
the {\it sharp} regularity up to the boundary. We refer to Section \ref{Sec5}
for  further comments on regularity in the case of convex domains.
The eigenvalue optimization problem \eqref{Comp} presents many interesting features, among which the study of preservation of symmetry is quite relevant, especially in view of the results proved in \cite{CGIKO00} for the {\it second order problem}. In particular, symmetry preserving properties are necessary conditions for uniqueness of optimal configurations. Indeed, if an optimal pair $(u,\rho)$ does not preserve the symmetry of the domain, it is always possible to realize the same double infimum $\Theta$ by taking the symmetric of $(u,\rho)$ as a different optimal pair. Clearly, uniqueness of optimal pairs in this context has to be considered up to a multiplicative constant for $u$, since by zero degree homogeneity of \eqref{Comp} (w.r.t. $u$), if $(u,\rho)$ is an optimal pair, for every constant $c\neq 0$, also $(c u,\rho)$ is  an optimal pair of the same problem.
In this perspective, in \cite{CoVe} we proved that if $\Omega=B$ is a ball, there exists
a unique optimal pair $(u,\rho)$, and $u$ is radial, positive and radially decreasing. As a consequence, the sublevel set $\{u \leq t\}$ appearing in \eqref{rho} is a ring-shaped region. \\
\indent In this paper, continuing the study started in \cite{CoVe}, we prove a rigidity result and a symmetry preservation property of optimal pairs for more general sets $\Omega$. 

Throughout the paper, given $x\in\partial \Omega$,  $\nu(x)$ denotes the external unit normal to $\Omega$ at $x$.

\indent We state first the rigidity result.
\begin{theorem}\label{Main1}
Let $\partial\Omega$ be $C^4$-smooth and let $(u,\rho)$ be an optimal pair for \eqref{Comp}. If $u$ satisfies the additional condition
\begin{equation}\label{Over}
\frac{\partial u}{\partial \nu}= c\quad \textrm{on } \partial \Omega\quad\mbox{for some }c<0,
\end{equation}
then $\Omega$ is a ball and $u$ is radially symmetric and radially decreasing.
\end{theorem}

Theorem \ref{Main1} deals with an overdetermined problem in the spirit of Serrin \cite{Serrin}
and Weinberger \cite{Weinberger}. The techniques used by Serrin and Weinberger in their papers are very different:
the first one is based on
the moving plane method and maximum principle, while the second relies
on integral identities and the construction of a suitable P-function.\\
Starting from the seminal papers \cite{Serrin,Weinberger}, many authors
started studying possible generalizations of those results for different
operators. In particular, there has been a certain attention devoted to fourth order overdetermined problems.
Let us briefly survey part of the results already available in the literature. 
In \cite{Troy81}, Troy proved a rigidity result for second order elliptic systems
imposing as many extra conditions as the number of the equations, namely  
\begin{equation}\label{Troy-sys}
\left\{ \begin{array}{rl}
            -\Delta u_i = f_i(u_1,\dots,u_m) & \textrm{in } \Omega,\smallskip\\
            u_i >0 & \textrm{in } \Omega,\smallskip\\
			u_i = 0 & \textrm{on } \partial \Omega,
\end{array}\right.\quad i=1,\dots,m.
\end{equation}
We observe that under suitable conditions on the $f_i$'s, for $m=2$, \eqref{Troy-sys} can
be formulated as a fourth order boundary value problem similar to the one we are interested in.
In \cite{Bennett}, Bennett considered a pure Dirichlet overdetermined problem,
namely 
\begin{equation*}
\left\{ \begin{array}{rl}
            \Delta^2 u = -1 & \textrm{in } \Omega,\smallskip\\
						u=\displaystyle{\frac{\partial u}{\partial \nu}} = 0 & \textrm{on } \partial \Omega,\smallskip\\
						\Delta u = c & \textrm{on } \partial \Omega,
\end{array}\right.
\end{equation*}
\noindent and showed that the existence of a solution $u \in C^{4}(\overline{\Omega})$
implies that $\Omega$ is a ball, constructing an appropriate P-function in the spirit
of Weinberger. 
In \cite{PaSc}, Payne and Schaefer considered the Navier version of
the problem considered by Bennett in \cite{Bennett}, in the case of a planar domain $\Omega$ which is star-shaped with
respect to the origin. The extension to any dimension for $C^2$  bounded domains was then provided in \cite{PhRa}
by Philippin and Ragoub, and subsequently by Goyal and Schaefer in
\cite{GoSc} with a different proof. 
In the present article, as in the aforementioned papers \cite{Troy81,PaSc,PhRa}, the
main tool for proving Theorem \ref{Main1} is the moving plane method. We stress here that
one of the keys that allows the use of such a classical technique is the monotonicity
result provided by Lemma \ref{4cases}, which is heavily based on \eqref{rho}.
\\

As already mentioned, the second main result concerns the preservation of symmetry; 
we extend to more general domains the results proved for the ball in \cite{CoVe}. In the next statement, we use the notation for the moving plane method recalled in Section \ref{Sec2}.
 
\begin{theorem}\label{Main2}
Let $\Omega$ be symmetric and convex  with respect to the hyperplane $\{x_1=0\}$, and
with $C^{4}$-smooth boundary $\partial \Omega$. Let $\lambda_1$ be as in \eqref{DefLambda1} and suppose that $\lambda_1 = 0$.
If $(u,\rho)$ is an optimal pair, then $u$ is symmetric with respect to $\{x_1=0\}$ and strictly decreasing in $x_1$ for $x_1 > 0$. 
Consequently, $\rho$ is symmetric with respect to the same hyperplane as well. 
Furthermore, the set $\{u >t\}$ is convex with respect to  $\{x_1=0\}$.
\end{theorem}

This theorem guarantees that optimal pairs preserve reflection symmetries of the domain, in presence of some convexity assumption. 
Since smooth strictly convex domains, and in particular balls, are covered
by the assumptions of Theorem \ref{Main2}, we note that \cite[Theorem~1.5]{CoVe} 
can be seen as a corollary of it.

Let us spend a few words concerning the hypotheses and the proof of Theorem \ref{Main2}.
$C^4$-smoothness (actually even just convexity) of $\Omega$ is a sufficient condition to ensure that
\eqref{P} can be {\it equivalently} written as the following second order {\it cooperative elliptic system} (we stress that the equivalence between weak solutions of \eqref{P} and weak solutions of \eqref{Sys} does not always hold, see Section \ref{Sec5})
\begin{equation}\label{Sys}
\left\{ \begin{array}{rl}
          -\Delta u_1 = u_2 & \quad \textrm{in } \Omega,\\  
          -\Delta u_2 = \Theta \rho  u_1 & \quad \textrm{in } \Omega,\\
					u_1= u_2 =0 & \quad \textrm{on } \partial \Omega.
					\end{array}\right.
\end{equation}
Due to its cooperative structure, \eqref{Sys} --and consequently the fourth order problem \eqref{P}-- inherits the strong maximum principle holding for second order elliptic equations. 
This allows us to apply the moving plane method for proving Theorem \ref{Main2}. 

\begin{figure}[h]
\includegraphics[scale=0.8]{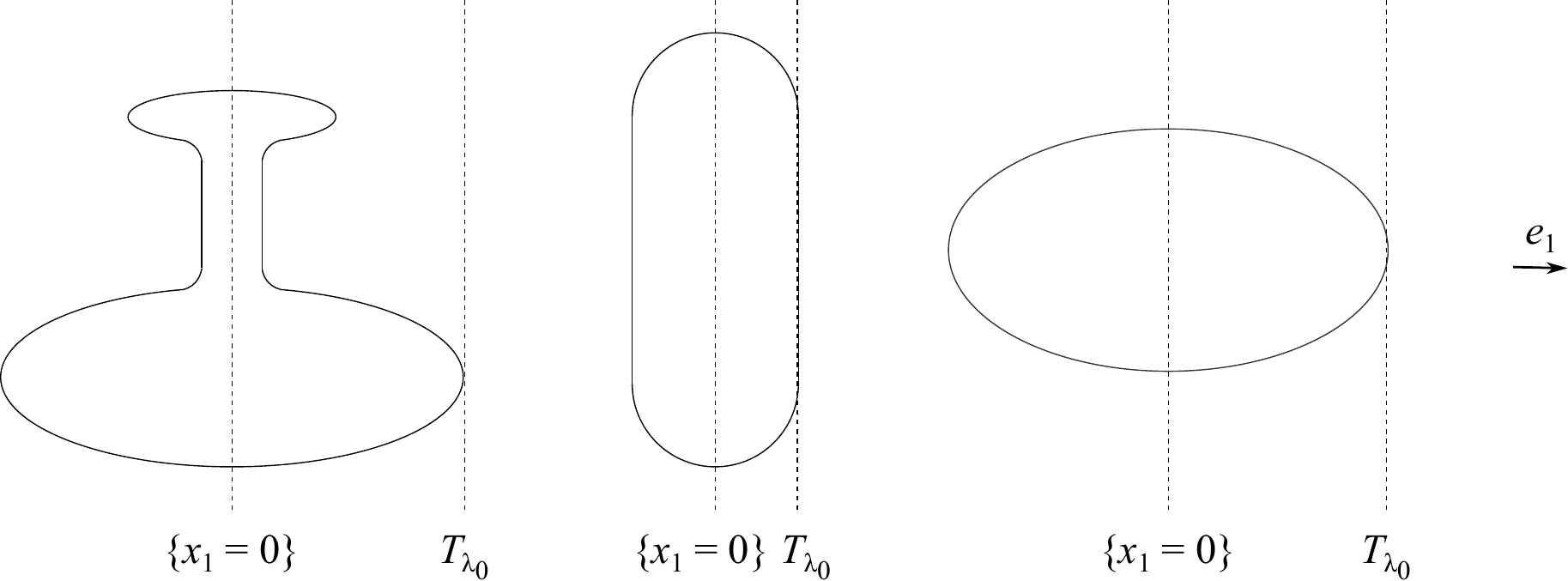}
\caption{Qualitative representation of some planar domains covered by Theorem \ref{Main2}. They are meant to be $C^4$-smooth and symmetric and convex with respect to $\{x_1=0\}$. Note that points with horizontal tangent or flat parts in the $e_1$ direction --i.e., orthogonal to the symmetry axis $\{x_1=0\}$-- are not allowed for $\partial \Omega\cap\{x_1 > 0\}$.
}\label{fig:domini}
\end{figure}

As already mentioned, the convexity of the domain is enough to have a maximum principle for \eqref{P}, cf. Section \ref{Sec5}; the main reason for requiring $\Omega$ to be $C^{4}$-smooth is to ensure $C^1$-regularity {\it up to the boundary} of both $u$ and $\Delta u$. This is needed in a series of lemmas proved in Section \ref{Sec3}, in the spirit of \cite{Troy81} by Troy
 and of \cite{GNN} by Gidas-Ni-Nirenberg, which are the core of the proof of Theorem \ref{Main2}. On the other hand, the assumption  
$\lambda_1=0$
is essentially due to the lack of regularity of $\rho_u$, which prevents any use of the Mean Value Theorem for applying the maximum principle also to the auxiliary cooperative elliptic system (see for instance \eqref{sys_w}) classically arising in the moving plane method. We want to stress that, while the hypothesis $\lambda_1=0$ seems to be mainly related to the technique used, a convexity hypothesis (at least w.r.t. $\{x_1=0\}$) in Theorem~\ref{Main2} seems to be crucial for the validity of the result. Indeed, some numerical evidence in \cite{KK} (cf. Fig.~6.9 therein) shows that symmetry breaking phenomena appear for thin annuli.\\

At this stage, it seems necessary to spend a few words regarding the possibility to relax the assumptions on $\Omega$.
In order to treat the case of more general domains, it would be quite natural to try to adapt the moving plane method as developed by Berestycki and Nirenberg
in \cite{BN}, for which it is enough to work with bounded sets. However, in our case, we already mentioned
that we need some further assumption (e.g. convexity) to ensure the equivalence between the fourth order PDE
and the second order elliptic system. In addition, we have some more considerations.
Essentially, the technique of Berestycki and Nirenberg is composed by two parts: the starting procedure and the continuation.
The first one is usually performed by means of {\it maximum principle in small domains},
in contrast with the smooth case, where one can use the Hopf Lemma.
In our case, thanks to the explicit expression \eqref{rho} of the optimal configuration $\rho_{u}$, we can
use the maximum principle in small domains for cooperative elliptic systems proved by de Figueiredo in \cite{deFigueiredo}, to let the moving plane method start. 
In principle, this technique would be the appropriate tool to deal
with more general open sets $\Omega$ also in the crucial part of the {\it continuation of the moving plane method},
without any further restrictive requirements. 
However, in our case, due to the lack of regularity of $\rho_u$,
it seems necessary to have more information on the
{\it structure} of the set $\{u \leq t\}$ ($(u,\rho_u)$ being optimal pair) or, better, its complementary in $\Omega$. In particular, if we know that
the set $\{u >t\}$ is symmetric and convex with respect to the hyperplane $\{x_1=0\}$, 
then we can adapt the proof of \cite[Theorem 1.3]{BN} to our setting, including sets $\Omega$ which are merely convex. This partial symmetry preserving result for (possibly non-smooth) convex domains $\Omega$ is the content of Proposition \ref{prop:convex}. Without any structure information on $\{u>t\}$ we are not able to prove, with this technique, symmetry preserving properties 
for sets having $\lambda_1 \neq 0$ but $\lambda_2 = 0$ (defined in \eqref{DefLambda2}),
see Figure~\ref{Fig2} and Remark \ref{RmkSec3}.\\

\smallskip

The structure of the paper is the following:
in Section \ref{Sec2} we fix the standing notation for implementing the
moving plane technique and recall some basic results from the literature
to make the paper self-contained. In Section \ref{Sec3} we state and prove some preliminary lemmas, while in Section \ref{Sec4} we prove Theorems \ref{Main1} and \ref{Main2}. Finally, in Section \ref{Sec5}, we treat the case of convex domains and prove Proposition \ref{prop:convex}.


\section{Notation and background}\label{Sec2}
In this section we recall some known results, prove preliminary lemmas, 
and introduce definitions and notation which will be useful in the next sections.
We stress that most of the results we are going to present in this section
are valid for general convex sets, see Remark~\ref{Rmk5}, but we state them in the less general context provided by Theorem~\ref{Main2}.

\medskip\medskip

\noindent{\bf Positivity.} We prove below some positivity results for optimal pairs. 
For sake of completeness, we start with a maximum principle that slightly
extends \cite[Lemma 1]{FGW}.
\begin{lemma}\label{max_principle}
Let $\Omega$ be a $C^4$-smooth bounded domain and set 
$$
\mathcal C^+:=\{v\in H^1_0(\Omega)\cap H^2(\Omega)\,:\, v\ge 0\mbox{ a.e. in }\Omega\}.
$$ 
Suppose that $u\in H^1_0(\Omega)\cap H^2(\Omega)$ satisfies 
\begin{equation}\label{forMP}
\int_\Omega \Delta u\Delta v\ge 0\quad\mbox{for every }v\in \mathcal C^+.
\end{equation}
Then, $u\in\mathcal C^+$. Furthermore, either $u\equiv 0$ or $u>0$ a.e. in $\Omega$.
\end{lemma}
\begin{proof}
Let $\varphi\in C_\mathrm{c}^\infty(\Omega)$, $\varphi\ge0$, and let $v_\varphi\in H^1_0(\Omega)$ be a solution of 
$$-\Delta v_\varphi=\varphi\quad\mbox{in }\Omega.$$
By elliptic regularity theory, $v_\varphi\in H^1_0(\Omega)\cap H^2(\Omega)$ and by maximum principle \cite[Theorem~8.1]{GT}, $v_\varphi\ge 0$ in $\Omega$. Hence, $v_\varphi\in\mathcal C^+$. Thus, using $v_\varphi$ as test function in \eqref{forMP}, we have
$$
0\le \int_\Omega \Delta u\Delta v_\varphi=-\int_\Omega \Delta u\, \varphi.
$$
By the arbitrariness of $\varphi$, we get again by \cite[Theorem 8.1]{GT} that $u\ge 0$ a.e. in $\Omega$. Finally, suppose that for some ball $B\Subset\Omega$ 
$$\inf_B u=\inf_\Omega u\ge 0.$$
Then, by \cite[Theorem 8.19]{GT}, $u\equiv \mathrm{Const.}$, and so $u\equiv 0$ a.e. in $\Omega$, being  $u\in H^1_0(\Omega)$.
\end{proof}

Using the previous lemma, it is possible to prove the following positivity result. 

\begin{proposition}\label{pos}
Let $\Omega$ be a $C^4$-smooth bounded domain, and $(u,\rho)$ an optimal pair, then $u>0$ and $\Delta u<0$ a.e. in $\Omega$.
\end{proposition}
\begin{proof}
Proceeding as in \cite[Proposition 5.1]{CoVe}, with $B$ replaced by $\Omega$ and using Lemma \ref{max_principle}, it is possible to prove that $u>0$ in $\Omega$. Now, setting $u_1:=u$ and considering system \eqref{Sys} with $\rho=\rho_{u_1}$, we get for $u_2:=-\Delta u_1$ 
$$
\begin{cases}
-\Delta u_2=\Theta\rho u_1\quad&\mbox{in }\Omega,\\
u_2=0&\mbox{on }\partial\Omega.
\end{cases}
$$
Hence, by the strong maximum principle \cite[Theorem 8.19]{GT}, $-\Delta u_1=u_2>0$ a.e. in $\Omega$.
\end{proof}

\medskip

\medskip
\noindent{\bf Notation.} We introduce now the notation for the moving plane technique. Given $x\in \R^n$, we denote by $x_1, \dots,x_n$ its components. 
For a given unit vector $e \in \R^n$ and for $\lambda\in\R$, we define the hyperplane
$$
T_{\lambda}:= \{ x \in \R^n: \scal{e}{x}=\lambda \}.
$$
From now on, unless otherwise stated, without loss of generality, we assume 
that $e=e_1$, i.e., the normal to $T_{\lambda}$ is parallel
to the $x_1$-direction. 
For every $\lambda\in\mathbb R$, we define the reflection with respect to $T_\lambda$ as the function $\varphi_\lambda:\,\mathbb R^n \to \mathbb R^n$ such that  
\begin{equation}\label{def:phi_lambda}
\varphi_\lambda(x_1,x_2,\dots,x_n)=(2\lambda-x_1,x_2,\dots,x_n)\quad\mbox{for every } x=(x_1,x_2,\dots,x_n)\in\mathbb R^n.
\end{equation}

We consider now the domain $\Omega$ where the problem is set. 
Given any $\lambda\in\mathbb R$, we introduce the (possibly empty) set 
\begin{equation*}
\Sl:=\{x\in\Omega\, : \,x_1>\lambda\}
\end{equation*}
and its reflection with respect to
$T_{\lambda}$, 
$$
\Spl:=\{\varphi_\lambda(x)\in \mathbb R^n \,:\,x\in\Sl\}.
$$

Since $\Omega \subset \R^n$ is bounded, we have that
$T_{\lambda}$ does not touch $\overline{\Omega}$ for sufficiently big $\lambda$.
Decreasing the value of the parameter $\lambda$, we can find a value $\lambda_0$ such 
that $T_{\lambda_0}$ touches $\overline{\Omega}$:
$$
\lambda_0:=\sup\{\lambda\in \mathbb R\,:\, T_\lambda\cap\overline{\Omega}\neq\emptyset\}.
$$
If we continue to lower the value of the parameter $\lambda$, we have that the hyperplane $T_{\lambda}$
cuts off from $\Omega$ the portion $\Sl$. Clearly, at the beginning of the process, the reflection $\Spl$ of $\Sl$ will be contained in $\Omega$.
Now, we define the value $\lambda_1$ as follows
\begin{equation}\label{DefLambda1}
\lambda_1:=\sup\{\lambda<\lambda_0\,:\,\mbox{(i) or (ii) is verified}\},
\end{equation}
where
\begin{itemize}
\item[(i)] $\Spl$ is internally tangent to the boundary $\partial \Omega$ at a certain point $P \notin T_{\lambda}$;
\item[(ii)] $T_{\lambda}$ is orthogonal to the boundary $\partial \Omega$ at a certain point $Q \in T_{\lambda} \cap \partial \Omega$.
\end{itemize}
By construction, 
$$
\Sigma'_{\lambda} \subset \Omega\quad\mbox{for every }\lambda\in[\lambda_1,\lambda_0).
$$ 
Nevertheless, by further decreasing the value of $\lambda$ below $\lambda_1$,
$\Sigma'_{\lambda}$ might still be contained in $\Omega$. Therefore we define
the value  
\begin{equation}\label{DefLambda2}
\lambda_2:=\inf\{\lambda<\lambda_0\,:\, \Sigma'_{\lambda} \subset \Omega \}.
\end{equation}
For every $i=1,2$ and $\lambda\in[\lambda_2,\lambda_0)$, we can define
\begin{equation}\label{def:w_i}
w_i^{(\lambda)} := u_i \circ \varphi_\lambda - u_i \quad \textrm{in } \Sl.
\end{equation}
We observe that, under the assumptions of Theorem \ref{Main2},
$\lambda_1=\lambda_2=0$, and so the definition \eqref{def:w_i} is well-posed for every $\lambda\in[0,\lambda_0)$.
 
Now, let $x_0\in \partial\Omega$ and $\ep >0$: 
we define the following sets
\begin{equation}\label{OmegaEp}
\Omega_{\ep} := \left\{x \in \Omega \,:\, \|x- x_0\| < \ep \right\},
\end{equation}
and 
\begin{equation}\label{SEp}
S_{\ep} := \left\{x \in \partial \Omega\,: \, \|x- x_0\| < \ep \right\}.
\end{equation}
\smallskip 

We finally present an auxiliary result which will be useful in the next sections. 

\begin{lemma}\label{4cases}
Let $\lambda\in[\lambda_2,\lambda_0)$, $(u,\rho_u)$ be an optimal pair, and $u\circ\varphi_{\lambda}-u\ge 0$ in $\Sl$. Then 
$$
(\rho_u\circ\varphi_\lambda)(u\circ\varphi_\lambda)-\rho_u u\ge 0\quad \mbox{in }\Sl.
$$
\end{lemma}
\begin{proof} Taking into account the explicit representation \eqref{rho} of $\rho_u$, pointwise there are four cases to treat.
\begin{itemize}
\item {\bf Case I: } $x,\varphi_\lambda(x) \in \{u \leq t\}\cap \Sl$. In this case,  
$$
\rho_{u}(\varphi_\lambda(x))\, u(\varphi_\lambda(x)) - \rho_{u}(x)\,u(x) = h \, (u(\varphi_{\lambda}(x))-u(x)) \geq 0.
$$
\item {\bf Case II: } $x \in \{u \leq t\}\cap \Sl$ and $\varphi_\lambda(x) \in \{u > t\}\cap \Sl$. In this case,  
$$
\begin{aligned}
\rho_{u}(\varphi_\lambda(x))\, u(\varphi_\lambda(x)) - \rho_{u}(x)\,u(x) &= H \, u(\varphi_\lambda(x)) - h\, u(\varphi_\lambda(x))\\
& \geq H (u(\varphi_{\lambda}(x))-u(x))\ge 0.
\end{aligned}
$$
\item {\bf Case III: } $x \in \{u > t\}\cap \Sl$ and $\varphi_\lambda(x) \in \{u \leq t\}\cap \Sl$. This case cannot occur because $u(\varphi_\lambda(x)) \geq u$ in $\Sl$.
\item {\bf Case IV: } $x,\varphi_\lambda(x) \in \{u > t\}\cap\Sl$. In this case,
$$
\rho_{u}(\varphi_\lambda(x))\, u(\varphi_\lambda(x)) - \rho_{u}(x)\,u(x) = H \, (u(\varphi_{\lambda}(x))-u(x)) \geq 0.
$$
\end{itemize}
Since these cases exhaust all possibilities, we are done.
\end{proof}


\section{Some preliminary lemmas}\label{Sec3}
We start the present section by proving some technical lemmas which are the adaptation to our setting of  
those proved by Gidas, Ni and Nirenberg in \cite{GNN} for second order problems, cf. also  Troy \cite[Lemma 4.1-4.3]{Troy81} for fourth order problems.
Note that, with respect to the results by Troy, we work with less regular solutions, i.e., the $u_i$'s satisfy the system in a weak sense. On the other hand, we can exploit the special form of the right hand side through the explicit representation of $\rho$ given in \eqref{rho}. 

\begin{lemma}\label{4.1}
Let $\partial\Omega\in C^4$,  $\lambda\in(\lambda_1,\lambda_0]$, $x_0 \in T_{\lambda}\cap\partial \Omega$, 
and $(u_1,u_2)$ be a weak solution of \eqref{Sys} with $\rho=\rho_{u_1}$.
Fix $\epsilon>0$ so small that $\nu_1(x)>0$ for every $x\in S_\epsilon$.
Then, there exists a positive constant $\delta>0$ such that 
$$
\dfrac{\partial u_i}{\partial x_1} < 0 \quad \textrm{in } \Omega_{\delta}\quad\mbox{for every }i=1,2.
$$
\end{lemma}
\begin{proof} We first observe that the existence of an $\epsilon$ as in the statement is ensured by the fact that $\nu_1(x_0)>0$ being $\lambda\in(\lambda_1,\lambda_0]$, and  by the regularity of $\Omega$. By \eqref{reg}, $u_i \in C^{1}(\overline{\Omega_{\ep}})$, $u_i >0$ in $\Omega_{\ep}$ and $u_i = 0 $ on $S_\epsilon$ for every $i=1,2$. 

Let $x\in S_{\ep}$. Since for every $i=1,2$, $u_i>0$ in $\Omega$ and 
$u_i=0$ on $S_{\ep}\subset \partial \Omega$,
\begin{equation}\label{partialnu}
\frac{\partial u_i}{\partial x_1}(x)=\lim_{t\to0^+}\dfrac{u_i(x)-u_i(x-t e_1)}{t}\leq 0
\end{equation}
for every $x\in S_\epsilon$ and $i=1,2$.
Arguing by contradiction, assume that there exists a sequence of points
$\{x_k\}_{k\in \N} \subset\Omega_{\ep}$ such that
\begin{equation}\label{abs}
\lim_{k\to\infty} x_k = x_0\quad\mbox{ and }\quad \frac{\partial u_i}{\partial x_1}(x_k) \geq 0\mbox{ for every }k\in\mathbb N\mbox{ and }i=1,2.
\end{equation}
Now, fix $k\in \N$ and consider the interval in the positive
$x_1$-direction having $x_k$ as left bound. This interval
must intersect $S_{\ep}$ at a certain point, denoted by $a_k$. In view of \eqref{partialnu},  $\tfrac{\partial u_i}{\partial x_1}(a_k) \leq 0$. Thus,
by \eqref{abs} and the regularity of the $u_i$'s, we have
\begin{equation}\label{dux1=0}
\dfrac{\partial u_i}{\partial x_1}(x_0) =0.
\end{equation}
Since $(u_1,u_2)$ solves weakly \eqref{Sys}, we get  
$$
-\Delta u_1 = u_2 > 0 \quad\mbox{and}\quad-\Delta u_2 = \Theta \, \rho_{u_1} \, u_1 >0\quad \textrm{in } \Omega_{\ep}.
$$
Finally, since $\partial\Omega\in C^4$, there exists an interior ball touching $\partial\Omega$ at $x_0$, hence by Hopf's Lemma (cf. \cite{PS}, \cite[Lemma 4]{Cellina}) we have $\dfrac{\partial u_i}{\partial \nu}(x_0)< 0$, and so 
$$
\dfrac{\partial u_i}{\partial x_1}(x_0) < 0,
$$
\noindent which contradicts \eqref{dux1=0} and concludes the proof.
\end{proof}

\begin{lemma}\label{4.2}
Let $\lambda \in [\lambda_1, \lambda_0)$ and $(u_1,u_2)$ be a weak solution of \eqref{Sys} with $\rho=\rho_{u_1}$. Suppose that 
\begin{equation}\label{Cond2}
u_i \leq u_i \circ \varphi_\lambda \quad \textrm{but } \quad u_i \not\equiv u_i \circ \varphi_\lambda \quad \textrm{in } \Sl\quad \mbox{ for some } i =1,2.
\end{equation}
Then,
$$
u_i<u_i \circ \varphi_\lambda \quad \textrm{in } \Sl \quad \mbox{ for every } i =1,2
$$
and 
$$\dfrac{\partial u_i}{\partial x_1} <0 \quad \textrm{on } \Omega \cap T_{\lambda}\quad \mbox{ for every } i =1,2.$$
\end{lemma}
\begin{proof} Fix $\lambda\in[\lambda_1,\lambda_0)$. Since $(u_1,u_2)$ solves weakly \eqref{Sys} we get
$$-\Delta (u_1 \circ \varphi_\lambda) = u_2 \circ \varphi_\lambda \quad \textrm{in }  \Sl,$$
\noindent and
$$-\Delta (u_2 \circ \varphi_\lambda) = \Theta \;(\rho_{u_1}\circ\varphi_\lambda) \;(u_1 \circ \varphi_\lambda) \quad \textrm{in } \Sl,$$
where we have used the invariance of $-\Delta$ with respect to the reflection $\varphi_\lambda$.
Now, by \eqref{def:w_i} and \eqref{Sys}, we have 
\begin{equation}\label{w_1}
-\Delta w_1^{(\lambda)} = u_2 \circ \varphi_\lambda - u_2 = w_2^{(\lambda)} \quad\textrm{in }\Sl,
\end{equation}
\noindent and 
\begin{equation}\label{w_2}
-\Delta w_2^{(\lambda)} = \Theta [ (\rho_{u_1}\circ\varphi_\lambda)\; (u_1 \circ \varphi_\lambda) -  \rho_{u_1}u_1]\quad\textrm{in }\Sl.
\end{equation}
If \eqref{Cond2} holds with $i=2$, by \eqref{w_1} we get $-\Delta w_1^{(\lambda)} \ge 0$ in $\Sl$. Furthermore,  
$$
w_1^{(\lambda)}=\begin{cases}
0 \quad&\mbox{on }T_\lambda,\\
u_1\circ\varphi_\lambda&\mbox{on }\partial\Sl\cap \partial\Omega,
\end{cases}
$$
and so $w_1^{(\lambda)}\ge 0 $ on $\partial\Sl$ being $\varphi_\lambda(\partial\Sl\cap \partial\Omega)\subset\overline{\Omega}$ and $u_1\ge 0$ in $\overline{\Omega}$. Hence, by the weak maximum principle \cite[Corollary 3.2]{GT},
we get that $w_1^{(\lambda)} \ge 0$ in $\Sl$. If \eqref{Cond2} holds with $i=1$, $w_1^{(\lambda)}\ge 0$ in $\Sl$, and by Lemma \ref{4cases}, we know that the right-hand side of \eqref{w_2} is non-negative. Hence, $-\Delta w_2^{(\lambda)} \geq 0$ in $\Sl$, and arguing as in the previous case, we get $w_2^{(\lambda)}\ge 0$ by the weak maximum principle. Therefore, in both cases \eqref{Cond2} holding with $i=1$ and $i=2$, we obtain 
$$
w_i^{(\lambda)} \ge 0 \quad\mbox{in }\Sl\quad\mbox{for every }i=1,2.
$$
Thus, by the strong maximum principle \cite[Theorem 2.2]{DP}, we get that for every $i=1,2$ 
$$
w_i^{(\lambda)}>0 \quad\mbox{ in }\Sl
$$ 
and 
$$
0>\frac{\partial w_i^{(\lambda)}}{\partial \nu}=-\frac{\partial w_i^{(\lambda)}}{\partial x_1}=2\frac{\partial u_i}{\partial x_1}\quad	\mbox{on } \Omega \cap T_{\lambda},
$$ 
where we have used the fact that  $w_i^{(\lambda)} = 0$ on $\Omega \cap T_{\lambda}$. 
This concludes the proof.
\end{proof}

\begin{remark}\label{symmetryOmega}
Under the assumptions of Lemma \ref{4.2}, if there exist $i=1,2$ and a point
$P \in \Omega \cap T_{\lambda_1}$ for which
$$\dfrac{\partial u_i}{\partial x_1}(P)=0,$$
then, {\it for every} $i=1,2$, $u_i$ is symmetric with respect to the hyperplane $T_{\lambda_1}$,
and 
$$
\Omega = \Sigma_{\lambda_1} \cup \Sigma'_{\lambda_1} \cup (\Omega \cap T_{\lambda_1}).
$$
Indeed, being $u_i\le u_i\circ\varphi_{\lambda_1}$ in $\Sigma_{\lambda_1}$, as a direct consequence of Lemma \ref{4.2}, we get 
$$
u_i \equiv u_i\circ \varphi_{\lambda_1}\quad \mbox{in }\Sigma_{\lambda_1}\quad\mbox{for every }i=1,2.
$$
This means that both $u_1$ and $u_2$ are symmetric with respect to the hyperplane $T_{\lambda_1}$.    
Now, since $u_i >0$ in $\Sigma_{\lambda_1}$ and $u_i=0$ on $\partial \Omega$,
it must be
$$
\Omega = \Sigma_{\lambda_1} \cup \Sigma'_{\lambda_1} \cup (\Omega \cap T_{\lambda_1}).
$$
Otherwise, it would exist a point $Q\in(\Omega\cap\partial\Sigma'_{\lambda_1})\setminus\left( \Sigma_{\lambda_1} \cup \Sigma'_{\lambda_1} \cup (\Omega \cap T_{\lambda_1})\right)$ where $u_i(Q)>0$, being $Q\in\Omega$, and simultaneously $u_i(Q)=u_i(\varphi_{\lambda_1}(Q))=0$, being $\varphi_{\lambda_1}(Q)\in\partial\Sigma_{\lambda_1}\setminus T_{\lambda_1}$. This is impossible and concludes the proof.
\end{remark}

\begin{lemma}\label{4.3}
Let $\partial \Omega\in C^4$, $\lambda \in (\lambda_1,\lambda_0)$, and $(u_1,u_2)$ be a weak solution of \eqref{Sys} with $\rho=\rho_{u_1}$.
Then
\begin{equation*}
\dfrac{\partial u_i}{\partial x_1} <0 \quad \textrm{and}\quad u_i < u_i \circ \varphi_\lambda \quad	\mbox{in } \Sl\quad\mbox{for every }i=1,2.
\end{equation*}
\end{lemma}
\begin{proof}
By Lemma \ref{4.1} there exists $\delta>0$ for which 
$$
\dfrac{\partial u_i}{\partial x_1}<0  \quad \textrm{in } \Omega_\delta.
$$
Hence, there exists $\bar \lambda\in(\lambda_1,\lambda_0)$ sufficiently close to $\lambda_0$ so that for both $i=1,2$
\begin{equation}\label{goal1}
\dfrac{\partial u_i}{\partial x_1} <0 \quad\mbox{in } \Sigma_{\lambda}
\end{equation}
for every $\lambda\in[\bar \lambda,\lambda_0)$.
In particular, there exists $\lambda\in(\bar\lambda,\lambda_0)$ close enough to $\lambda_0$ for which $\Sigma'_{\lambda}\subset \Sigma_{\bar\lambda}$. Therefore, $\dfrac{\partial u_i}{\partial x_1} <0$ in $\Sigma'_{\lambda}\cup\Sl \cup (T_{\lambda}\cap \Omega)$ for every $i=1,2$. Being the $u_i$'s of class at least $C^1(\Omega)$, the Mean Value Theorem guarantees that for every $x\in\Sl$ there exists $\xi_x\in \{t\varphi_{\lambda}(x)+(1-t) x\;:\; t\in(0,1)\}\subset\Sigma'_{\lambda}\cup\Sl \cup (T_{\lambda}\cap \Omega)$ such that 
$$
0>\frac{\partial u_i}{\partial x_1}(\xi_x)=\frac{u_i(x)-u_i(\varphi_{\lambda}(x))}{(x-\varphi_{\lambda}(x))_1}.
$$ 
Hence, 
\begin{equation}\label{goal}
u_i < u_i\circ\varphi_{\lambda} \quad \textrm{in } \Sl\quad\mbox{for every }i=1,2
\end{equation}
and for $\lambda$ sufficiently close to $\lambda_0$. 
We need to gain some more space where \eqref{goal1} and \eqref{goal} hold. To this aim, we lower $\lambda$ until
we reach the threshold $\mu \geq \lambda_1$, which is defined as follows:
\begin{equation}\label{DefMu}
\mu := \inf \left\{\lambda \in (\lambda_1,\lambda_0): \; \frac{\partial u_i}{\partial x_1}<0\mbox{ and }u_i < u_i\circ\varphi_{\lambda}\textrm{ hold in } \Sl \mbox{ for every }i=1,2
\right\}.
\end{equation} 
Thus, for $\lambda = \mu$ we have by continuity
\begin{equation}\label{At_Mu}
\frac{\partial u_i}{\partial x_1}<0 \; \mbox{ and }\;	u_i \leq u_i\circ\varphi_\mu \quad \textrm{in } \Sigma_\mu \quad\mbox{for every }i=1,2.
\end{equation}
There are now two cases to treat. Either $\mu > \lambda_1$ or $\mu = \lambda_1$.
Let us start with the case $\mu > \lambda_1$. We claim that, in this case
\begin{equation}\label{At_Mu2}
u_i\not\equiv u_i\circ\varphi_\mu\quad\mbox{in }\Sigma_\mu \quad\mbox{for every }i=1,2.
\end{equation}
Indeed, if $u_i\equiv u_i\circ\varphi_\mu$ in $\Sigma_\mu$ for some $i=1,2$, then by the continuity of $u_i$ up to the boundary of $\Omega$, we would have $u_i(x)=u_i(\varphi_\mu(x))$ for every $x\in\partial\Sigma_\mu\setminus T_\mu$. On the other hand, since $\mu>\lambda_1$, $\varphi_\mu(\partial\Sigma_\mu\setminus T_\mu)\subset \Omega$. Hence , by Propositions \ref{pos} and \ref{equivalence}
$$
u_i(x)=0<u_i(\varphi_\mu(x))\quad\mbox{for every }x\in \partial\Sigma_\mu\setminus T_\mu, 
$$
which is absurd and proves the claim. 
Therefore, in view of conditions \eqref{At_Mu} and \eqref{At_Mu2}, we can apply Lemma \ref{4.2} to get  
\begin{equation}\label{fromLe2}
u_i<u_i\circ\varphi_\mu\mbox{ in }\Sigma_\mu \quad\mbox{and}\quad \frac{\partial u_i}{\partial x_1}<0\mbox{ in }\Omega\cap T_\mu\quad\mbox{for every }i=1,2. 
\end{equation}
By the definition \eqref{DefMu} of $\mu$, we know that at least one of the following two cases holds:
\begin{itemize}
\item[(a)] there exist $\epsilon>0$ and a sequence of points $(x_k)\subset\Sigma_{\mu-\epsilon}\setminus\Sigma_\mu$ converging to some point $x\in T_\mu$ such that 
\begin{equation}\label{Opposite1}
\frac{\partial u_i}{\partial x_1}(x_k) \geq 0\quad	\mbox{for some }i=1,2\mbox{ and for all }k\in\mathbb N;
\end{equation}
\item[(b)] there exist $\epsilon>0$ and an increasing sequence of numbers $(\lambda_k)_{k \in \mathbb{N}} \subset (\mu-\epsilon, \mu)$ converging to $\mu$, such that for every $k \in \mathbb{N}$, there is a point $x_k \in \Sigma_{\lambda_k}\setminus\Sigma_\mu$ for which
\begin{equation}\label{Opposite}
 u_i(x_k) \geq u_i(\varphi_{\lambda_k}(x_k))\quad	\mbox{for some }i=1,2\mbox{ and for all }k\in\mathbb N.
\end{equation}
\end{itemize}

If (a) holds, clearly $\frac{\partial u_i}{\partial x_1}(x)\ge 0$ for the same index $i$ as in \eqref{Opposite1}. Hence, by \eqref{fromLe2}, $x\in\partial\Omega$. Since $\mu >\lambda_1$, $T_\mu$ is not orthogonal to $\partial\Omega$ at $x$ and so $\nu_1(x)>0$. In view of Lemma~\ref{4.1}, this implies that $\frac{\partial u_i}{\partial x_1}< 0$ in a neighborhood of $x$, a contradiction. \\
\indent If (b) holds, then, up to a subsequence, $x_k\to x\in \overline{\Sigma_\mu}$. Hence, $\varphi_{\lambda_k}(x_k)\to \varphi_\mu (x)$. Consequently, by the continuity of $u_i$, 
$$
u_i(x)\ge u_i(\varphi_\mu(x))\quad\mbox{ for the same index $i$ as in \eqref{Opposite}}
$$
which, in view of \eqref{fromLe2}, gives $x\in\partial\Sigma_\mu$. If $x\in\partial\Sigma_\mu\setminus T_\mu$, then $\varphi_\mu(x)\in\Omega$, being $\mu>\lambda_1$. This forces 
$$
0=u_i(x)<u_i(\varphi_\mu(x)),
$$
which is absurd. 
Thus, it must be $x\in\partial\Sigma_\mu\cap T_\mu$. By the $C^1$-regularity of $u_i$ and the Mean Value Theorem, for $k\in\mathbb N$ sufficiently large there exists $$\xi_k\in \{t\varphi_{\lambda_k}(x_k)+(1-t) x_k\;:\; t\in(0,1)\}\subset\Sigma_{\mu-\epsilon}\setminus\Sigma_\mu$$ for which 
$$
\frac{\partial u_i}{\partial x_1}(\xi_k)=\frac{u_i(\varphi_{\lambda_k}(x_k))-u_i(x_k)}{(\varphi_{\lambda_k}(x_k)-x_k)_1}\ge 0,
$$
where in the last inequality we have used \eqref{Opposite} and the fact that $x_k\in\Sigma_{\lambda_k}$.
In this way, we have built a sequence $(\xi_k)$ verifying (a), but we have already proved that this is not possible. 
This excludes the possibility that $\mu>\lambda_1$. 

The only remaining possibility is that $\mu = \lambda_1$. Hence, for every $\lambda\in (\lambda_1,\lambda_0)$ the condition \eqref{goal} holds.
Reasoning as above, we obtain 
$$
u_i\not \equiv u_i\circ\varphi_\lambda\quad\mbox{in }\Sigma_\lambda\quad\mbox{for every }i=1,2,
$$
for every $\lambda\in(\lambda_1,\lambda_0)$. So, in particular, by Lemma \ref{4.2}, $\tfrac{\partial u_i}{\partial x_1}<0$ on $\Omega\cap T_\lambda$, which by the arbitrariness of $\lambda\in (\lambda_1,\lambda_0)$ gives that $\frac{\partial u_i}{\partial x_1}<0$ in the whole $\Sigma_{\lambda_1}$ for every $i=1,2$ and closes the proof.
\end{proof}

\section{Proofs of Theorems \ref{Main1} and \ref{Main2}}\label{Sec4}
In order to prove the rigidity result stated in Theorem \ref{Main1}, we take inspiration essentially from
\cite{PhRa} by Philippin and Ragoub, and \cite[Theorem 2]{Troy81} by Troy.

\begin{proof}[$\bullet$ Proof of Theorem \ref{Main1}] We start fixing an arbitrary direction, say $e_1$, along which we will move a hyperplane from infinity towards $\Omega$. By the discussion in the Introduction, we know that $(u_1,u_2):=(u,-\Delta u)$ is a weak solution of \eqref{Sys} with $\rho=\rho_{u_1}$. 
We consider the functions $w_i^{(\lambda_1)}$ defined in \eqref{def:w_i}, namely
$$
w_i^{(\lambda_1)} = u_i \circ \varphi_{\lambda_1} - u_i \quad \textrm{in } \Sigma_{\lambda_1}\quad\mbox{for every }i=1,2.
$$
By Lemma \ref{4.3}, we know that 
$$
w_{i}^{(\lambda_1)} \geq 0 \quad \textrm{in } \Sigma_{\lambda_1}\quad\mbox{for every }i=1,2.
$$
Hence, by Lemma \ref{4.2}, either  
\begin{equation}\label{i}
w_{i}^{(\lambda_1)} >0 \quad \textrm{in } \Slo\quad\mbox{for every }i=1,2
\end{equation}
or
\begin{equation}\label{ii}
w_i^{(\lambda_1)}\equiv 0 \quad\mbox{in }\Sigma_{\lambda_1}\quad\mbox{for every } i=1,2.
\end{equation}
Now, we observe that if the latter holds, the proof of the theorem is completed. Indeed, \eqref{ii} implies in particular that $u=u_1$ and $-\Delta u=u_2$ are symmetric with respect to the
hyperplane $T_{\lambda_1}$. Furthermore, being $u>0$ in $\Omega$ by \cite[Proposition 5.1]{CoVe},
and $u=0$ on $\partial\Omega$, it must be (cf. Remark \ref{symmetryOmega})
$$
\Omega=\Sigma_{\lambda_1}\cup\Sigma'_{\lambda_1}\cup\left(\Omega\cap T_{\lambda_1}\right),
$$
that is $\Omega$ is symmetric with respect to $T_{\lambda_1}$ as well. By the arbitrariness of the direction fixed, we infer that $\Omega$ must be a ball and $u$ radially symmetric. The monotonicity property of the solution $u$ then follows from \cite[Theorem 1.5]{CoVe}. 
It remains to exclude the case \eqref{i}. 

To this aim, suppose first that we are in case (i) of \eqref{DefLambda1},
namely that $\Sigma'_{\lambda_1}$ is internally tangent to $\partial \Omega$
at a point $P \notin T_{\lambda_1}$. It is clear that, $P':=\varphi_{\lambda_1}(P)\in\partial\Sigma_{\lambda_1}\setminus T_{\lambda_1}$, hence $P,\,P'\in\partial\Omega$, and so
$$
u_{1}(P) = u_1(P') = 0.
$$
Therefore, arguing as in the proof of Lemma \ref{4.2}, taking into account \eqref{i} and the fact that $w_1^{(\lambda_1)}\in C^2(\Sigma_{\lambda_1})\cap C(\overline{\Sigma_{\lambda_1}})$, we deduce that $w_1^{(\lambda_1)}$ satisfies the following problem
\begin{equation*}
\begin{cases}
&-\Delta w_1^{(\lambda_1)}=w_2^{(\lambda_1)}>0 \quad \mbox{ in }\Sigma_{\lambda_1},\\
&w_1^{(\lambda_1)}>0\mbox{ in }\Sigma_{\lambda_1}, \; w_1^{(\lambda_1)}(P')=0.
\end{cases}
\end{equation*}
Thus, by the Hopf Lemma \cite[Lemma 4]{Cellina}, we obtain
\begin{equation*}
\dfrac{\partial w_{1}^{(\lambda_1)}}{\partial \nu}(P')<0.
\end{equation*}
On the other hand, by straightforward calculations we have 
$$
\displaystyle{\frac{\partial w_1^{(\lambda_1)}}{\partial \nu}}(P')=0. 
$$
Indeed, by \eqref{Over} 
$$
\frac{\partial w_1^{(\lambda_1)}}{\partial \nu}(P')=
\frac{\partial (u\circ\varphi_{\lambda_1})}{\partial \nu}(P') - \frac{\partial u}{\partial \nu}(P')= \frac{\partial u}{\partial \nu}(P)-\frac{\partial u}{\partial \nu}(P')=c-c=0.
$$
This excludes case (i).\\
Let us now assume that we are in case (ii) of \eqref{DefLambda1}.
We want to show that the function $w_{1}^{(\lambda_1)}$ has a zero of
order two at the point $Q \in \partial \Omega \cap T_{\lambda_1}$, namely
that both the first and the second derivatives of $w_{1}^{(\lambda_1)}$
vanish at $Q$. This is achieved performing the same computations
as in \cite[pp. 307-308]{Serrin}. We stress that we can replicate that argument 
thanks to the $C^{3,\alpha}(\overline{\Omega})$-regularity of $u_1$, cf. \eqref{reg}.
Actually, that argument is local in nature, so it is enough to note that 
in our situation we are even more smooth in a neighborhood of the boundary.
We then reach a contradiction using the {\it corner lemma} of Serrin \cite[Lemma 1]{Serrin}. 
This excludes the case \eqref{i} and closes the proof.
\end{proof}

\begin{remark}
The techniques used to prove Theorem \ref{Main1} actually allow to 
slightly relax the assumptions. Following \cite[Theorem 1]{Praja},
we can prescribe the normal derivative $\tfrac{\partial u}{\partial \nu}$
to be constant on $\partial \Omega$ except on one point where
either an interior or exterior sphere condition holds.
Since the proof of \cite[Theorem 1]{Praja} is mainly based only on 
geometric considerations on the set $\Omega$, that argument can be applied to our case as well.
\end{remark}

\medskip

\begin{proof}[$\bullet$ Proof of Theorem \ref{Main2}]
As usual, being $(u,\rho)$ an optimal pair, under any of the two conditions (i) or (ii) in the statement, we can consider the couple $(u_1,u_2):=(u,-\Delta u)$ as a weak solution of \eqref{Sys} with $\rho=\rho_{u_1}$.
We start observing that, in order to prove the theorem, it is enough to show that 
\begin{equation}\label{tohold}
\mbox{for every }\lambda\in(0,\lambda_0)\quad\mbox{ it holds }\quad u_i<u_i\circ\varphi_\lambda\mbox{ in } \Sl\mbox{ for every }i=1,2.
\end{equation}
and 
\begin{equation}\label{tohold1}
\mbox{for every }-\lambda\in(-\lambda_0,0)\quad\mbox{ it holds }\quad u_i<u_i\circ\varphi_{-\lambda}\mbox{ in } \varphi_0(\Sigma_{\lambda})\mbox{ for every }i=1,2.
\end{equation}

Indeed, by \eqref{tohold} and by the definition \eqref{def:phi_lambda} of $\varphi_\lambda$, for every $\lambda\in(0,\lambda_0)$, if we take $x=(x_1,y)\in\Sl$, i.e., $x\in\Omega$ with $x_1>\lambda$, we get  
$$
u_i(x_1,y)<u_i(2\lambda-x_1,y)\quad\mbox{for every }i=1,2.
$$
Hence, $(2\lambda-x_1,y)\to (-x_1,y)$ as $\lambda\searrow 0$, and by continuity
$$
u_i(x_1,y)\le u_i(-x_1,y)\quad\mbox{for every }i=1,2.
$$

On the other hand, for every $-\lambda\in(-\lambda_0,0)$, if we take $x=(-x_1,y)\in\varphi_0(\Sl)$, i.e., $x\in\Omega$, with $-x_1<-\lambda$, we get by \eqref{tohold1}
$$
u_i(-x_1,y)<u_i(-2\lambda+x_1,y)\quad\mbox{for every }i=1,2.
$$
So, $(-2\lambda+x_1,y)\to (x_1,y)$ as $-\lambda\nearrow 0$, and by continuity
$$
u_i(-x_1,y)\le u_i(x_1,y)\quad\mbox{for every }i=1,2.
$$
Altogether, for any $x=(x_1,y)\in\Omega$
$$
u_i(x_1,y)= u_i(-x_1,y)\quad\mbox{for every }i=1,2,
$$
that is $u=u_1$ and $-\Delta u=u_2$ are symmetric with respect to $\{x_1=0\}$. 
Moreover, if for every $\lambda\in(0,\lambda_0)$ we define $w_i^{(\lambda)}$ as in \eqref{def:w_i}, by \eqref{tohold} we know that $w_i^{(\lambda)}>0$ in $\Sigma_\lambda$ for every $i=1,2$. Therefore,  for every $\lambda\in(0,\lambda_0)$ it holds
\begin{equation*}
\begin{cases}
\begin{aligned}
-\Delta w_1^{(\lambda)}=w_2^{(\lambda)}>0 & \quad \mbox{in } \Sigma_{\lambda},\\
w_1^{(\lambda)}>0 & \quad \mbox{in }\Sigma_\lambda,\\
 w_1^{(\lambda)}= 0 &\quad \mbox{on } T_\lambda.
\end{aligned}
\end{cases}
\end{equation*}
Then, by Hopf's Lemma (cf. \cite[Lemma 4]{Cellina}) we obtain 
$$
0>\frac{\partial w_1^{(\lambda)}}{\partial \nu}=-\frac{\partial w_1^{(\lambda)}}{\partial x_1}=2\frac{\partial u_1}{\partial x_1} \quad\mbox{on } T_\lambda,
$$
where $\nu$ denotes the unit outward normal to $\partial\Sigma_\lambda$. This means that  $\dfrac{\partial u}{\partial x_1}<0$ for $x_1>0$, as required. The convexity of the set $\{u>t\}$ with respect to the hyperplane $\{x_1=0\}$ is then a consequence of the same property holding for the domain $\Omega$ and of the monotonicity of $u$. 

By the assumptions on $\Omega$, $\lambda_1=0$. Hence, the validity of \eqref{tohold} is guaranteed by Lemma~\ref{4.3}. The case of \eqref{tohold1} is analogous, due to the symmetry of the problem. 
\end{proof}


\section{Convex domains}\label{Sec5}
In this section we deal with convex sets $\Omega$.
In particular, this implies that $\partial\Omega$ is Lipschitz continuous, cf. \cite[Corollary 1.2.2.3]{Grisvard}.\\
As already mentioned in the Introduction, setting $u_1:=u$, the equivalence between problem \eqref{P} and system \eqref{Sys} is a key tool for our purposes. Nevertheless, such equivalence is not straightforward when the domain is not smooth enough (i.e., $C^4$). 
When dealing with a fourth order problem
of the form
\begin{equation}\label{PDE4}
\left\{ \begin{array}{rl}
             \Delta^2 u = f & \textrm{in } \Omega,\\
						 u=\Delta u = 0 & \textrm{on } \partial \Omega,
				\end{array}\right.		
\end{equation}
beyond the usual weak solutions, we can define the so-called {\it system solutions}, see e.g. \cite[Section 2.7]{GGS}. We recall below both definitions.  
\begin{definition}
Let $f \in L^{2}(\Omega)$, we say that $u$ is a {\it weak solution} 
of \eqref{PDE4}, if $u \in H^{2}(\Omega)\cap H^{1}_{0}(\Omega)$ and
$$\int_{\Omega}\Delta u \, \Delta \varphi \, dx = \int_{\Omega} f \varphi \, dx \quad \textrm{for every } \varphi \in H^{2}(\Omega)\cap H^{1}_{0}(\Omega).$$ 
\end{definition}
\begin{definition}
Let $f \in L^{2}(\Omega)$, we say that $u$ is a {\it system solution} 
of \eqref{PDE4}, if $u \in H^{1}_{0}(\Omega)$, $\Delta u \in H^{1}_{0}(\Omega)$
and $u$ solves (weakly) the system
\begin{equation}\label{Pf}
\left\{ \begin{array}{rl}
              -\Delta u = v & \textrm{in } \Omega,\\
						  -\Delta v = f & \textrm{in } \Omega,\\
							u=v=0 & \textrm{on } \partial \Omega.
				\end{array}\right.							
\end{equation}
\end{definition}
These two types of solutions are in general different, as it is shown by the
Sapondzyan paradox or the Babuska paradox, see \cite[Chapter 2, Section 7]{GGS}.
One common feature of these two paradoxes, is the fact that the set $\Omega \subset \R^{n}$
on which they are considered, admit concave corners. 
On the other hand, if one assumes $\Omega$ to be convex, these kinds of phenomena disappear, and one can prove that system solutions coincide
with weak solutions. 
\begin{proposition}[Corollary 1.6 of \cite{NaSw}]
Let $\Omega$ be convex and $f \in L^2(\Omega)$. Then $u$ is a weak solution of \eqref{Pf} if and only if $u$ is a system solution of the same problem.
\end{proposition}

Furthermore, the following result holds. 
\begin{proposition}[\cite{Kadlec} and Theorem 1.2 of \cite{NaSw}]\label{H2}
Let $\Omega$ be convex. Then for every $f \in L^2(\Omega)$ the (unique) weak solution $u\in H^1_0(\Omega)$ of 
$$
\begin{cases}
-\Delta u = f\quad&\mbox{in }\Omega,\\
u=0&\mbox{on }\partial\Omega
\end{cases}
$$ 
belongs to $H^2(\Omega)$.
\end{proposition}

Altogether, we can summarize the previous considerations as follows.
\begin{proposition}\label{equivalence}
Let $\Omega$ be a convex bounded domain, then the following two sentences are equivalent
\begin{itemize}
\item[(i)] $(u,\rho)$ is an optimal pair of \eqref{Comp};
\item[(ii)] $(u_1,u_2):=(u,-\Delta u)$ is a weak solution of \eqref{Sys} with $\rho=\rho_{u_1}$ as in \eqref{rho}.
\end{itemize}
\end{proposition} 
\smallskip 

\begin{remark}\label{Rmk5}
We want to stress that, thanks to Proposition \ref{H2}, Lemma \ref{max_principle} holds even for convex domains. 
Therefore this allows to follow verbatim the proof of Proposition \ref{pos}, showing the positivity of $u$
when $\Omega$ is a convex set without any further regularity assumption. Namely, the following statement holds true: 

\quote{\it Let $\Omega$ be a convex bounded domain and $(u,\rho)$ an optimal pair, then $u>0$ and $\Delta u<0$ a.e. in $\Omega$.}
\end{remark}
\smallskip

We recall below a particular instance of the maximum principle for cooperative systems in small domains
proved for a more general situation by de Figueiredo in \cite{deFigueiredo}.

\begin{lemma}[Proposition 1.1 of \cite{deFigueiredo}]\label{MP_coop_sys}
Let $\Omega\subset\mathbb R^n$ be a bounded domain,
$$
\mathcal L:=\left(
\begin{array}{ll}
\Delta & 0\\
0&\Delta
\end{array}
\right)\quad \mbox{and}\quad
A(x):=\left(
\begin{array}{ll}
a_{11}(x) &a_{12}(x)\\
a_{21}(x) &a_{22}(x)
\end{array}
\right),
$$
with $a_{ij}\in L^\infty(\Omega)$ for $i,j=1,2$ and $a_{ij}\ge 0 $ for $i\neq j$. Suppose that $U:=(u_1,u_2)$ satisfies
$$
\begin{cases}
&-(\mathcal L+A)U\ge0\quad\mbox{in }\Omega,\\
&\displaystyle{\liminf_{x\to\partial\Omega}}\,U(x)\ge0.
\end{cases}
$$
Then, there exists $\delta=\delta(n,\mathrm{diam}(\Omega))>0$ such that if $|\Omega|<\delta$, $U\ge 0$ in $\Omega$ (i.e, $u_1\ge 0$ and $u_2\ge0$ in $\Omega$). 
\end{lemma}

\smallskip
We must now spend a few words regarding the regularity of $u$. 
We recall that if $\Omega$ is $C^4$,
by \cite[Theorem 1.3]{CoVe} we know that if $(u,\rho)$ is an optimal pair, 
$u\in C^{3,\alpha}(\overline\Omega)\cap W^{4,q}(\Omega)$ for every $\alpha\in(0,1)$ and $q\ge1$.
Since here $\Omega$ is required to be just convex, only Lipschitz regularity of the boundary is guaranteed. In particular, $\Omega$ satisfies an
exterior cone condition and optimal interior regularity can be deduced as in \cite[Theorem~1.3(a)]{CoVe}, using a bootstrap argument and embedding theorems. We get in this way 
\begin{equation*}
u \in  C^{3,\alpha}(\Omega)\cap W^{4,q}_{\mathrm{loc}}(\Omega)\quad\textrm{for all } \alpha \in (0,1) \textrm{ and } q \geq 1,
\end{equation*}
cf. \cite[Remark 3.2]{CoVe}.

We are interested also in the regularity {\it up to the boundary} of optimal
pairs in the case of convex domains. In this regard, we can prove the following lemma. 

\begin{lemma}	
Let $\Omega$ be convex and $(u,\rho)$ be an optimal pair. Then, both $u$ and $\Delta u$ are of class $C^{0,\gamma}(\overline\Omega)$ for some $\gamma\in(0,1)$.
\end{lemma}
\begin{proof} We set as usual $u_1:= u$ and we consider the system \eqref{Sys} with $\rho=\rho_{u_1}$, i.e.,
\begin{equation*}
\left\{ \begin{array}{rl}
           -\Delta u_1 = u_2 & \textrm{in } \Omega,\\
					 -\Delta u_2 =\Theta \rho u_1 & \textrm{in } \Omega,\\
					 u_1= u_2=0 & \textrm{on } \partial \Omega.
				\end{array}\right.
\end{equation*}
Since $\Theta \rho u_1 \in L^{2}(\Omega)$, by Proposition \ref{H2}
we know that $u_1,\,u_2 \in H^{2}(\Omega)\cap H^{1}_{0}(\Omega)$.
Moreover, by Remark \ref{Rmk5} and Proposition \ref{equivalence}, we also know that $u_1, u_2 >0$ a.e. in $\Omega$. 
Now, without loss of generality, we can assume that $\Theta \rho \leq K$ for
some positive constant $K>1$ and by \eqref{Sys}
we get
\begin{equation}\label{IneqSys}
\left\{ \begin{array}{rl}
           -\Delta u_1 < K u_2 & \textrm{in } \Omega,\\
					 -\Delta u_2 \leq K u_1 & \textrm{in } \Omega,\\
					 u_1= u_2=0 & \textrm{on } \partial \Omega.
				\end{array}\right.
\end{equation} 	
Let us now consider the function $\varphi:= u_1 + u_2$. 
By \eqref{IneqSys}, we obtain				
\begin{equation*}
\left\{ \begin{array}{rl}
           -\Delta \varphi \leq K \varphi & \textrm{in } \Omega,\\
					 \varphi=0 & \textrm{on } \partial \Omega.
				\end{array}\right.
\end{equation*} 
Clearly, $\varphi \in H^{2}(\Omega)\cap H^{1}_{0}(\Omega)$, being $u_1$ and $u_2$ in the same space. Hence,
by \cite[Theorem~8.15]{GT} with $L = \Delta +K$, we get that
$$\sup_{\Omega} \varphi \leq C \|\varphi\|_{L^{2}(\Omega)}$$
for some $C=C(n,K,|\Omega|)>0$ independent of $\varphi$.
Therefore, $\varphi \in L^{\infty}(\Omega)$,
and since both $u_1 \leq \varphi$ and $u_2 \leq \varphi$, this gives
$$
u_1, u_2 \in L^{\infty}(\Omega).
$$
In particular, $u_1,u_2 \in L^{q}(\Omega)$, with $q > n/2$, and we can now exploit \cite[Theorem 8.29]{GT} to get that
\begin{equation*}
u_1, u_2 \in C^{0,\gamma}(\overline{\Omega}) \quad \textrm{for some } \gamma \in (0,1)
\end{equation*}
and conclude the proof.
\end{proof}

We are now ready to prove our result for convex domains. If $(u,\rho)$ is an optimal pair, by the explicit form \eqref{rho} of $\rho=\rho_u$, we know that $\rho_u$ inherits all symmetry properties of $u$. The following proposition states that, under suitable convexity assumptions, also the converse is true, namely also $u$ inherits the symmetries of $\rho_u$. The argument used in the proof of the next proposition is in the spirit of the paper \cite{BN} by  Berestycki and Nirenberg, where a second order problem was treated, by using a careful combination of the maximum principle for small domains, the strong maximum principle in its classical form, and the moving plane technique. In the following proposition we use again the notation related to the moving plane technique, introduced in Section \ref{Sec2}.
 
\begin{proposition}\label{prop:convex}
Let $\Omega$ be a convex domain of $\mathbb R^n$ and suppose that $\Omega$ is symmetric with respect to the hyperplane $\{x_1=0\}$. 
Let $(u,\rho)$ be an optimal pair such that the superlevel set $\{u>t\}$ is symmetric and convex with respect to  $\{x_1=0\}$. \\
Then, $u$ is symmetric with respect to $\{x_1=0\}$ and strictly decreasing in $x_1$ for $x_1 > 0$. 
\end{proposition}
\begin{proof} 
As in the proof of Theorem \ref{Main2}, it is enough to prove \eqref{tohold}.
Let $w_i^{(\lambda)}$ be defined as in \eqref{def:w_i}. For every $\lambda\in (\lambda_1, \lambda_0)$, we get
\begin{equation}\label{sys_w}
\begin{cases}
-(\Delta w_1^{(\lambda)}+w_2^{(\lambda)})=0\quad&\mbox{in }\Sl,\\
-\{\Delta w_2^{(\lambda)}+\Theta [(\rho_{u_1}\circ\varphi_\lambda)(u_1\circ\varphi_\lambda)-\rho_{u_1}u_1] \}=0&\mbox{in }\Sl,\\
w_1^{(\lambda)}\ge0,\quad w_1^{(\lambda)}\not\equiv 0&\mbox{on }\partial\Sl,\\
w_2^{(\lambda)}\ge0,\quad w_2^{(\lambda)}\not\equiv 0&\mbox{on }\partial\Sl.
\end{cases}
\end{equation}
Now, the second equation in \eqref{sys_w} can be equivalently rewritten as
$$
-[\Delta w_2^{(\lambda)}+\Theta(\rho_{u_1}\circ\varphi_\lambda)w_1^{(\lambda)}]=\Theta(\rho_{u_1}\circ\varphi_\lambda-\rho_{u_1})u_1\quad	\mbox{in }\Sl.
$$
By  condition $u_1=0$ on $\partial\Omega$ and the $C^{0,\gamma}(\overline\Omega)$-regularity of $u_1$, we know that the set $\{u_1\le t\}$ contains a tubular neighborhood $\mathcal U$ of $\partial\Omega$. Hence, if $\lambda$ is sufficiently close to $\lambda_0$, $\Sl\subset \mathcal U$ and by the expression \eqref{rho} of $\rho_{u_1}$, we get $\rho_{u_1}\equiv h$ in $\Sigma_\lambda$. This implies 
$$
\rho_{u_1}\circ\varphi_\lambda-\rho_{u_1}\ge 0\quad\mbox{in }\Sl.
$$
Therefore, being $u_1>0$ in $\Omega$ by Remark \ref{Rmk5} and Proposition \ref{equivalence}, for $\lambda<\lambda_0$ sufficiently close to $\lambda_0$, we have
\begin{equation*}
\begin{cases}
-(\Delta w_1^{(\lambda)}+w_2^{(\lambda)})=0\quad&\mbox{in }\Sl,\\
-[\Delta w_2^{(\lambda)}+\Theta(\rho_{u_1}\circ\varphi_\lambda)w_1^{(\lambda)}]\ge 0&\mbox{in }\Sl,\\
w_1^{(\lambda)}\ge0,\quad w_1^{(\lambda)}\not\equiv 0&\mbox{on }\partial\Sl,\\
w_2^{(\lambda)}\ge0,\quad w_2^{(\lambda)}\not\equiv 0&\mbox{on }\partial\Sl.
\end{cases}
\end{equation*}
We can now apply the maximum principle in narrow domains for cooperative systems Lemma \ref{MP_coop_sys} in $\Sl$, with 
$$
A:=\left(
\begin{array}{ll}
0 &1\\
\Theta\rho_{u_1}\circ\varphi_\lambda &0
\end{array}
\right)
$$
and $U=(w_1^{(\lambda)},w_2^{(\lambda)})$, to get $w_i^{(\lambda)}\ge 0$ in $\Sl$ for every $i=1,2$. Therefore, for $\lambda$ close enough to $\lambda_0$, $w_1^{(\lambda)}$ and $w_2^{(\lambda)}$ satisfy respectively
$$
\begin{cases}
-\Delta w_1^{(\lambda)}=w_2^{(\lambda)}\ge 0\;&\mbox{in }\Sl,\\
w_1^{(\lambda)}\ge 0,\quad w_1^{(\lambda)}\not\equiv 0 &\mbox{on }\partial\Sl
\end{cases}\;\mbox{ and }\;
\begin{cases}
-\Delta w_2^{(\lambda)}=\Theta (\rho_{u_1}\circ\varphi_\lambda) w_1^{(\lambda)}\ge 0\;&\mbox{in }\Sl,\\
w_2^{(\lambda)}\ge 0,\quad w_2^{(\lambda)}\not\equiv 0 &\mbox{on }\partial\Sl.
\end{cases}
$$
Thus, by the strong maximum principle \cite[Theorem 8.19]{GT}, 
we get 
\begin{equation}\label{toholdBN}
w_{i}^{(\lambda)} > 0 \quad \textrm{in } \Sigma_{\lambda}\quad\mbox{for every }i=1,2
\end{equation}
\noindent and $\lambda<\lambda_0$ sufficiently close to $\lambda_0$.

We can now define $\mu_1$ as 
$$
\mu_1:=\inf\big\{\lambda\in(0,\lambda_0)\,:\,u_i<u_i\circ\varphi_\lambda  \mbox{ in } \Sl\mbox{ for every }i=1,2\big\}.
$$
By \eqref{toholdBN}, $\mu_1$ is well defined and $\mu_1 \geq 0$.
We want to prove that $\mu_1=0$. Suppose by contradiction that $\mu_1>0$. By continuity, 
\begin{equation*}
u_i\le u_i\circ\varphi_{\mu_1} \mbox{ in } \Sigma_{\mu_1}\quad\mbox{ for every }i=1,2.
\end{equation*}
Moreover, since $\mu_1>0$, there exists a point $x_1\in \partial \Sigma_{\mu_1}\setminus T_{\mu_1}$
such that $\varphi_{\mu_1}(x_1) \in \Omega$, and hence
$w_1^{(\mu_1)}(x_1)>0$. This shows that
$w_1^{(\mu_1)}\not\equiv 0$ on $\partial \Sigma_{\mu_1}$.
Therefore, we also have that $w_1^{(\mu_1)}$ solves the following problem
\begin{equation*}
\begin{cases}
-\Delta w_1^{(\mu_1)}=w_2^{(\mu_1)}\ge 0\quad&\mbox{in }\Sigma_{\mu_1},\\
w_1^{(\mu_1)}\ge 0, \quad w_1^{(\mu_1)}\not\equiv 0&\mbox{on }\partial\Sigma_{\mu_1}.
\end{cases}
\end{equation*}
Hence, the strong maximum principle \cite[Theorem 8.19]{GT} 
guarantees that 
$w_1^{(\mu_1)}>0$ in $\Sigma_{\mu_1}$. 
Similarly, we get that $w_2^{(\mu_1)}>0$ in $\Sigma_{\mu_1}$.\\
We want to prove that, if $\epsilon > 0$ is small,
\begin{equation}\label{wi-mu-eps}
w_i^{(\mu_1-\epsilon)}\ge 0\quad\mbox{in }\Sigma_{\mu_1-\epsilon}\quad\mbox{for every }i=1,2.
\end{equation}
This, together with the sign $w_i^{(\mu_1-\epsilon)}\ge 0$, $w_i^{(\mu_1-\epsilon)}\not\equiv 0$ on $\partial\Sigma_{\mu_1-\epsilon}$, would give by the strong maximum principle
\begin{equation*}
w_i^{(\mu_1-\epsilon)}>0\;\;\mbox{ in }\Sigma_{\mu_1-\epsilon}\quad\mbox{for every }i=1,2, 
\end{equation*}
and would contradict the minimality of $\mu_1$, showing that actually $\mu_1=0$. \\
In order to prove \eqref{wi-mu-eps}, let $K\Subset \Sigma_{\mu_1}$ be a compact set, then $w_i^{(\mu_1)}>0$ in $K$ for every $i=1,2$. 
We now claim that there exits $\epsilon_0>0$ so small that for every $\epsilon\in(0,\epsilon_0)$:
\begin{equation}\label{inK}
w_i^{(\mu_1-\epsilon)}>0\;\;\mbox{in }K\quad\mbox{for every }i=1,2.
\end{equation}
Indeed, let 
$$M_i := \min_{x \in K} \left\{ u_i \circ \varphi_{\mu_1}(x) - u_i(x) \right\} >0.$$
By continuity of the $u_i$'s, there exists $\epsilon \ll 1$ such that
$$|u_i \circ \varphi_{\mu_1 - \epsilon}(x)-u_i \circ \varphi_{\mu_1}(x)| \leq \tfrac{M_i}{2} \quad \textrm{for every } x \in K.$$
Hence, for every $x \in K$ and $i=1,2$,
\begin{equation*}
w_i^{(\mu_1-\epsilon)}(x)= u_i \circ \varphi_{\mu_1 - \epsilon}(x) - u_i \circ \varphi_{\mu_1}(x)+u_i \circ \varphi_{\mu_1}(x)-u_i(x) \geq \tfrac{M_i}{2}>0,
\end{equation*}
which proves the claim \eqref{inK}.

\begin{figure}[h]
\includegraphics[scale=0.8]{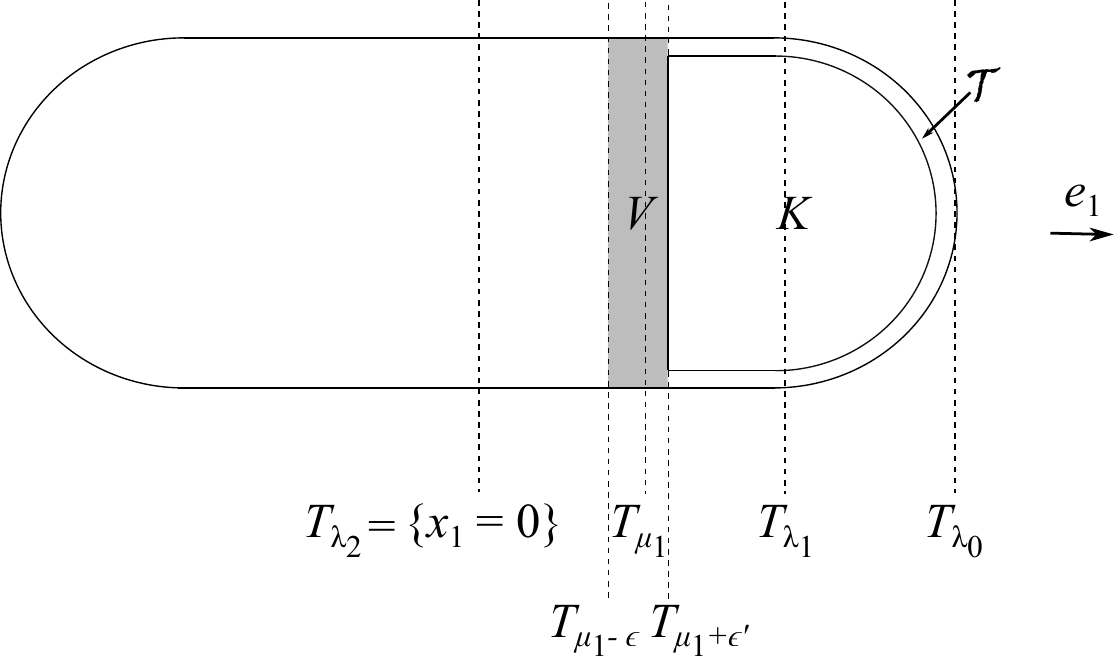}
\caption{Qualitative representation of the domain $\Omega$, its optimal cap $\Sigma_{\mu_1}$ in the assumption by contradiction, and its subsets $K$, $\mathcal T$ and $V$ defined in the proof.}\label{Fig2}
\end{figure}

In particular, we choose $K$ such that the complement of $K$ in $\Sigma_{\mu_1+\epsilon'}$ is contained in a tubular neighborhood of $\partial\Omega$, more precisely  
$$
\mathcal T:= \Sigma_{\mu_1+\epsilon'}\setminus K=\{x\in \Sigma_{\mu_1+\epsilon'}: \mathrm{dist}(x,\partial \Omega)<\epsilon '\}
$$
for some $\epsilon'>0$. 
Therefore, if we denote by $V:=\Sigma_{\mu_1-\epsilon}\setminus\Sigma_{\mu_1+\epsilon'}$, we get
$$
\Sigma_{\mu_1-\epsilon}=K\cup\mathcal T\cup V.
$$
Our aim is to apply the maximum principle for narrow domains in $\Sigma_{\mu_1-\epsilon}\setminus K$, cf. Figure \ref{Fig2}. Taking $\epsilon$ and $\epsilon'$ sufficiently small, we have $|\Sigma_{\mu_1-\epsilon}\setminus K|<\delta$, where $\delta=\delta(n,\mathrm{diam}(\Sigma_{\mu_1-\epsilon}\setminus K))>0$ is given in Lemma \ref{MP_coop_sys}.

In $\Sigma_{\mu_1-\epsilon}\setminus K$, $(w_1^{(\mu_1-\epsilon)}, w_2^{(\mu_1-\epsilon)})$ satisfies the following cooperative system 
$$
\begin{aligned}
&\begin{cases}
-\left(\Delta w_1^{(\mu_1-\epsilon)}+w_2^{(\mu_1-\epsilon)}\right)= 0&\mbox{ in }\Sigma_{\mu_1-\epsilon}\setminus K,\\
-\left(\Delta w_2^{(\mu_1-\epsilon)}+\Theta (\rho_{u_1}\circ\varphi_{\mu_1-\epsilon})w_1^{(\mu_1-\epsilon)}\right)=\Theta\left(\rho_{u_1}\circ\varphi_{\mu_1-\epsilon}-\rho_{u_1}\right)u_1&\mbox{ in }\Sigma_{\mu_1-\epsilon}\setminus K,\\
w_1^{(\mu_1-\epsilon)}\ge 0,\quad w_1^{(\mu_1-\epsilon)}\not\equiv0&\mbox{on }\partial\left(\Sigma_{\mu_1-\epsilon}\setminus K\right)\\
w_2^{(\mu_1-\epsilon)}\ge 0,\quad w_2^{(\mu_1-\epsilon)}\not\equiv0&\mbox{on }\partial\left(\Sigma_{\mu_1-\epsilon}\setminus K\right).
\end{cases}
\end{aligned}
$$
In order to apply Lemma \ref{MP_coop_sys}, we need to prove that 
\begin{equation}\label{sign}\rho_{u_1}\circ\varphi_{\mu_1-\epsilon}-\rho_{u_1}\ge 0\quad\mbox{ in }\Sigma_{\mu_1-\epsilon}\setminus K.
\end{equation} 
This is obviously true in $\mathcal T$, since for $\epsilon'$ sufficiently small, $\rho_{u_1}\equiv h$ in $\mathcal T$, by \eqref{rho}, being $u_1$ continuous in $\overline\Omega$ and $u_1=0$ on $\partial\Omega$.
We now consider $x\in V$. If $u_1(x)\le t$, $\rho_{u_1}(x)=h$ and so $\rho_{u_1}(\varphi_{\mu_1-\epsilon}(x))-\rho_{u_1}(x)\ge0$. If $u_1(x)> t$, we want to show that also $u_1(\varphi_{\mu_1-\epsilon}(x))>t$. Indeed, since $x\in\{u_1>t\}$ and $\{u_1>t\}$ is symmetric and convex with respect to $\{x_1=0\}$, 
$$
\overline{xx^0}:=\{\xi x+(1-\xi)\varphi_0(x)\,:\, \xi\in[0,1]\}\subset\{u_1>t\}.
$$ 
On the other hand, for $\epsilon$ sufficiently small, $\mu_1-\epsilon>0$, and so 
$$
|x-\varphi_{\mu_1-\epsilon}(x)|=2\mathrm{dist}(x,T_{\mu_1-\epsilon})=2[x_1-(\mu_1-\epsilon)]<2x_1=|x-\varphi_0(x)|, 
$$
that is to say $\varphi_{\mu_1-\epsilon}(x)$ belongs to the segment $\overline{xx^0}$. Thus, $u_1(\varphi_{\mu_1-\epsilon}(x))>t$, which gives $\rho_{u_1}(\varphi_{\mu_1-\epsilon}(x))-\rho_{u_1}(x)=0$.
This proves \eqref{sign} and by Lemma \ref{MP_coop_sys} gives 
$$
w_i^{(\mu_1-\epsilon)}\ge 0\;\;\mbox{in }\Sigma_{\mu_1-\epsilon}\setminus K\quad\mbox{for every }i=1,2.
$$
Together with \eqref{inK}, we get \eqref{wi-mu-eps}, which concludes the proof of the proposition.
\end{proof}

\begin{remark}\label{RmkSec3}
We stress that even for the composite membrane problem, the aforementioned convexity of $\{u>t\}$, provided $\Omega$
is convex as well, is proved with some extra assumptions in \cite{CGIKO00} and just conjectured for the general case.
We believe it would be interesting to address this issue in a future work.
\end{remark}



\end{document}